\documentclass{amsart}

\usepackage{comment}
\usepackage{amssymb}
\usepackage{graphicx}
\usepackage{amssymb}
\usepackage{epstopdf}
\usepackage{enumerate}
\usepackage{cite}
\usepackage{amsmath,amssymb,amsfonts,amsthm,amscd} 
\usepackage{graphics}                 
\usepackage{color}                    
\usepackage{hyperref} 
\usepackage{pictexwd,dcpic}
\usepackage{empheq}
\usepackage[OT2,T1]{fontenc}
\DeclareSymbolFont{cyrletters}{OT2}{wncyr}{m}{n}
\DeclareMathSymbol{\Sha}{\mathalpha}{cyrletters}{"58}

\newtheorem{theorem}{Theorem}[section]
\newtheorem{lemma}[theorem]{Lemma}
\newtheorem{proposition}[theorem]{Proposition}
\newtheorem{corollary}[theorem]{Corollary}
\newtheorem{fact}[theorem]{Fact}
\newtheorem{question}[theorem]{Question}
\newtheorem{problem}[theorem]{Problem}

\theoremstyle{remark}
\newtheorem{remark}[theorem]{Remark}

\theoremstyle{definition}

\theoremstyle{conjecture}
\newtheorem{conjecture}[theorem]{Conjecture}

\theoremstyle{definition}
\newtheorem{example}[theorem]{Example}

\numberwithin{equation}{section}

\def\F{\mathbb{ F}}
\def\Q{\mathbb{ Q}}
\def\C{\mathbb{ C}}
\def\Z{\mathbb{ Z}}

\def\P{\mathbb{ P}}

\def\CC{\mathcal{C}}
\def\Gal{\textup{\mbox{Gal}}}

\def\Nm{\mbox{Nm}}

\def\Hom{\mbox{Hom}}

\def\mp{\mathfrak{p}}

\def\Jac{\textup{\mbox{Jac}}}
\def\End{\textup{\mbox{End}}}
\def\GL{\textup{\mbox{GL}}}
\def\SL{\textup{\mbox{SL}}}
\def\PGL{\textup{\mbox{PGL}}}
\def\PSL{\textup{\mbox{PSL}}}
\DeclareMathOperator{\Xsplit}{\emph{X}_{\text{s}}}
\DeclareMathOperator{\Xnonsplit}{\emph{X}_{\text{ns}}}
\DeclareMathOperator{\Jsplit}{\emph{J}_{\text{s}}}
\DeclareMathOperator{\Jnonsplit}{\emph{J}_{\text{ns}}}
\DeclareMathOperator{\Cs}{C_{\text{s}}}
\DeclareMathOperator{\Csplus}{C_{\text{s}}^+}
\DeclareMathOperator{\Cns}{C_{\text{ns}}}
\DeclareMathOperator{\Cnsplus}{C_{\text{ns}}^+}
\def\Magma{\textsf{MAGMA}}
\def\Sage{\textsf{Sage}}

\begin{document}

\title[Tetrahedral Elliptic Curves]{Tetrahedral Elliptic Curves and the local-global principle for Isogenies}

\author{Barinder Singh Banwait}
\address{Institut de Mathe\'ematiques de Bordeaux, Universit\'e de
  Bordeaux I, 33405 Talence, France}
\email{Barinder.Banwait@math.u-bordeaux1.fr}
\thanks{The first author was supported by an EPSRC Doctoral Training
  Award at the University of Warwick}

\author{John Cremona}
\address{Mathematics Institute, University of Warwick, Coventry CV4 7AL, UK}
\email{J.E.Cremona@warwick.ac.uk}
\urladdr{http://homepages.warwick.ac.uk/staff/J.E.Cremona/}

\begin{abstract}
We study the failure of a local-global principle for the existence of
$l$-isogenies for elliptic curves over number fields~$K$.  Sutherland
has shown that over $\Q$ there is just one failure, which occurs for
$l=7$ and a unique $j$-invariant, and has given a classification of
such failures when $K$ does not contain the quadratic subfield of the
$l$'th cyclotomic field.  In this paper we provide a classification of
failures for number fields which do contain this quadratic field, and
we find a new `exceptional' source of such failures arising from the
exceptional subgroups of $\PGL_2(\F_l)$. By constructing models of two
modular curves, $\Xsplit(5)$ and $X_{S_4}(13)$, we find two new
families of elliptic curves for which the principle fails, and we show
that, for quadratic fields, there can be no other exceptional
failures.
\end{abstract}

\maketitle

\section{Introduction}\label{sec:intro}
Let $E$ be an elliptic curve defined over a number field~$K$, and $l$
a prime. It is easy to show that, if $E$ possesses a $K$-rational
$l$-isogeny, then the reduction $\tilde{E}_\mp/\F_\mp$, for all
primes~$\mp$ of~$K$ of good reduction and not dividing $l$, likewise
possesses an $\F_\mp$-rational $l$-isogeny.

In \cite{Drew}, Andrew Sutherland asked a converse question: if
$\tilde{E}_\mp/\F_\mp$ admits an $\F_\mp$-rational $l$-isogeny for a
density one set of primes $\mp$, then does $E/K$ admit a $K$-rational
$l$-isogeny?  Sutherland showed that, while the answer to this
question is usually ``yes'', there nevertheless exist pairs $(E/K,l)$ for which the
answer is ``no''.

Whether an elliptic curve over a field possesses a rational
$l$-isogeny or not depends only on its $j$-invariant, provided that
the $j$-invariant is neither~$0$ nor~$1728$; thus, if the answer is
``no'' for one elliptic curve~$E/K$ for the prime $l$, it is also
``no'' for every elliptic curve over $K$ with the same $j$-invariant
$j(E)$ (with the same exceptions). Following Sutherland, we thus
define a pair $(l,j_0)$, consisting of a prime~$l$ and an
element~$j_0\not=0,1728$ of a number field $K$, to be
\textbf{exceptional for $K$} if there exists an elliptic curve $E$
over $K$, with $j(E) = j_0$, such that the answer to the above
question at $l$ is ``no''. We will refer to the prime in the
exceptional pair as an \textbf{exceptional prime for $K$}, and any
elliptic curve $E$ over $K$ with $j(E) = j_0$ as a \textbf{Hasse at $l$ curve over $K$}.

Sutherland gives a necessary condition for the existence of an
exceptional pair, under a certain assumption. To state Sutherland's
result, recall that the absolute Galois group $G_K := \Gal(\bar{K}/K)$
acts on the $l$-torsion subgroup $E(\bar{K})[l]$, yielding the
{mod-$l$ representation} \[ \bar{\rho}_{E,l} : G_K \to \GL_2(\F_l),\]
whose image $G_{E,l} := \textup{Im }\bar{\rho}_{E,l}$ is well-defined
up to conjugacy; we refer to $G_{E,l}$ as \emph{the mod-$l$ image of
  $E$}. We let $H_{E,l} := G_{E,l}$ modulo scalars, and observe that
$H_{E,l}$ depends only upon $j(E)$, provided that $j(E) \neq 0$ or 1728;
we refer to $H_{E,l}$ as \emph{the projective mod-$l$ image of $E$}.

It is easy to show that $l=2$ is not an exceptional prime for any number field, so henceforth we assume that $l$ is odd. We now define $l^\ast := \pm l$, where the plus sign is taken if $l \equiv 1\pmod 4$, and the minus sign otherwise. 

Sutherland's result may now be stated as follows; by $D_{2n}$ we mean the dihedral group of order $2n$.

\begin{proposition}[Sutherland]
\label{Suther}
Assume $\sqrt{l^\ast} \notin K$. If $(l,j_0)$ is exceptional for $K$,
then for all elliptic curves~$E/K$ with $j(E)=j_0$,
\begin{enumerate}
\item
The projective mod-$l$ image of $E$ is isomorphic to $D_{2n}$, where
$n > 1$ is an odd divisor of $\frac{l-1}{2}$;
\item
$l \equiv 3 \pmod 4$;
\item
The mod-$l$ image of $E$ is contained in the normaliser of a split Cartan subgroup of $\GL_2(\F_l)$;
\item
$E$ obtains a rational $l$-isogeny over $K(\sqrt{l^\ast})$.
\end{enumerate}
\end{proposition}
(In fact, the converse is also true, as may be shown by applying the proof of the converse part of Proposition~\ref{P3} below; see Section~\ref{sec:pf-gp}.)

Sutherland used this result for $K=\Q$ to determine the exceptional
pairs for $\Q$ (where the assumption~$\sqrt{l^\ast} \notin\Q$ is
trivially satisfied for all $l$). If $(l,j(E))$ is exceptional for
$\Q$, then (3) above says that $E$ corresponds to a $\Q$-point on the
modular curve $\Xsplit(l)$. By the recent work of Bilu, Parent and
Rebolledo \cite{BPR}, it follows that $l$ must be 2,3,5,7 or 13. Of
these, only 3 and $7$ are $3\pmod{4}$, and $3$ can easily be ruled out
as a possible exceptional prime (for all number fields). Thus, $7$ is
the only possible exceptional prime for $\Q$, and (1) above tells us
that the projective mod-$7$ image of a Hasse at $7$ curve over $\Q$
must be isomorphic to $D_6$, the dihedral group of order~$6$. The
modular curve parametrising elliptic curves with this specific level-7
structure turns out to be the rational elliptic curve with label 49a3
in \cite{Cre}, which has precisely two non-cuspidal rational
points. Evaluating $j$ at these points yields the same value, and
hence gives Sutherland's second result.

\begin{theorem}[Sutherland, Theorem 2 in \cite{Drew}]
\label{Suth2}
$(7,\frac{2268945}{128})$ is the only exceptional pair for $\Q$. 
\end{theorem}

\noindent In this paper we would like to investigate what happens in the case where $\sqrt{l^\ast} \in K$. In Section~\ref{sec:pf-gp} we will prove the following using Sutherland's methods.

\begin{proposition}
\label{P3}
Assume $\sqrt{l^\ast} \in K$. Then $(l,j_0)$ is exceptional for $K$ if
and only if one of the following holds for elliptic curves~$E/K$ with
$j(E)=j_0$:
\begin{itemize}
\item
$H_{E,l} \cong A_4$ and $l \equiv 1\pmod{12}$;
\item
$H_{E,l} \cong S_4$ and $l \equiv 1\pmod{24}$;
\item
$H_{E,l} \cong A_5$ and $l \equiv 1\pmod{60}$;
\item
$H_{E,l} \cong D_{2n}$ and $l \equiv 1\pmod{4}$, where $n > 1$ is a divisor of $\frac{l-1}{2}$, and $G_{E,l}$ lies in the normaliser of a split Cartan subgroup. 
\end{itemize}
\end{proposition}

Thus, in the case $\sqrt{l^\ast} \in K$, there are two sorts of exceptional pairs; the \textbf{dihedral ones}, and the \textbf{non-dihedral ones}. 

Let us now consider each of these two cases over $K = \Q(\sqrt{l^\ast})$, the smallest field containing $\sqrt{l^\ast}$. Regarding the dihedral pairs, we may ask the following question.

\begin{question}
\label{Q3}
For which $l \equiv 1\pmod 4$ is there an elliptic curve $E$ over
$\Q(\sqrt{l})$ such that $H_{E,l} \cong D_{2n}$, for $n > 1$ a divisor
of $\frac{l-1}{2}$?
\end{question}

A positive answer to the Serre Uniformity Problem for number fields
would imply that there should be only finitely many such $l$, but we
are unable to prove this. Instead, we show that the set of $l$ asked
for by the above question is not empty; $l=5$ gives a positive
answer.

\begin{theorem}
\label{Thm1}
An elliptic curve $E$ over $\Q(\sqrt{5})$ has $H_{E,5} \cong D_4$ if
and only if its $j$-invariant is given by the formula
\begin{equation}\label{eqn:jmapV4}
  j(E) = \frac{((s+5)(s^2 - 5)(s^2 + 5s+10))^3}{(s^2 +  5s + 5)^5}
\end{equation}
for some $s \in \Q(\sqrt{5})$, together with the condition that $s^2 -
20$ is not a square in $\Q(\sqrt{5})$ for all $s \in \Q(\sqrt{5})$
satisfying \eqref{eqn:jmapV4}.

Thus, the exceptional pairs at 5 over $\Q(\sqrt{5})$ are given by $(5,j(E))$ for $j(E)$ as above, and in particular, there are infinitely many exceptional pairs at 5 over $\Q(\sqrt{5})$.
\end{theorem}

The proof of this theorem considers the modular curve $\Xsplit(5)$ corresponding to the normaliser of a split Cartan subgroup, whose
$\Q(\sqrt{5})$ points (as we will see) correspond to elliptic curves
$E$ over $\Q(\sqrt{5})$ with $H_{E,5} \subseteq D_4$. This curve is
defined over $\Q$ and has genus~$0$; writing the $j$-map \[\Xsplit(5)
\overset{j}\longrightarrow X(1)\] as a rational function yields the
parametrisation \eqref{eqn:jmapV4}; the further condition stated in the theorem is needed to force the
corresponding elliptic curve to have $H_{E,5} \cong D_4$ (and not merely
a subgroup of~$D_4$); see Section~\ref{sec:pf-thm} for the full proof.

Regarding the non-dihedral pairs, we prove the following in
Section~\ref{sec:pf-prop}. 
\begin{proposition}
\label{P4}
The only non-dihedral exceptional prime $l$ over any quadratic field
is~$13$ over~$\Q(\sqrt{13})$, where the projective mod-$13$ image is
isomorphic to $A_4$.
\end{proposition}

This leads to the following question. 

\begin{question}
\label{Q2}
Find all elliptic curves $E$ over $\Q(\sqrt{13})$ such that $H_{E,13}
\cong A_4$.
\end{question}

By Proposition~\ref{P4}, such elliptic curves are the only
non-dihedral Hasse curves over quadratic fields.

We take a similar approach to this question as we did for Theorem~\ref{Thm1}, by studying the relevant modular curve $X_{S_4}(13)$; this is the modular curve over $\Q$ corresponding to the pullback to $\GL_2(\F_{13})$ of $S_4 \subset \PGL_2(\F_{13})$; the earliest reference to this curve we are aware of is in Mazur's article \cite{MazurRat}. This modular curve is geometrically connected, and over the complex numbers has the description $\Gamma_{A_4}(13) \backslash \mathcal{H}^\ast$, where $\Gamma_{A_4}(13)$ is the pullback to $\PSL_2(\Z)$ of $A_4 \subset \PSL_2(\F_{13})$. A $\Q$-point on $X_{S_4}(13)$ corresponds to an elliptic curve $E/\Q$ such that $H_{E,13} \subseteq S_4$. A $\Q(\sqrt{13})$-point corresponds to an elliptic curve $E/\Q(\sqrt{13})$ such that $H_{E,13} \subseteq A_4$. Thus, the elliptic curves we seek in Question~\ref{Q2} correspond to certain $\Q(\sqrt{13})$-points on the modular curve $X_{S_4}(13)$. 

\begin{theorem}
\label{Thm2}
The modular curve $X_{S_4}(13)$ is a genus~$3$ curve, whose canonical embedding in $\P^2_\Q$ has the following model:
\begin{align*}
\CC: 4X^3Y - 3X^2Y^2 + 3XY^3 - X^3Z + 16X^2YZ - 11XY^2Z + \\ 5Y^3Z + 3X^2Z^2 + 9XYZ^2 + Y^2Z^2 + XZ^3 + 2YZ^3 = 0.
\end{align*}
On this model, the $j$-map $X_{S_4}(13) \overset{j}\longrightarrow X(1)$ is given by \[j(X,Y,Z) = \frac{n(X,Y,Z)}{d(X,Y,Z)^{13}},\] where 
\begin{align*}
d(X,Y,Z) &= 5 X^{3} - 19 X^{2} Y - 6 X Y^{2} + 9 Y^{3} + X^{2} Z \\ 
& - 23 X Y Z - 16 Y^{2} Z + 8 X Z^{2} - 22 Y Z^{2} + 3 Z^{3}
\end{align*}
and $n(X,Y,Z)$ is an explicit degree $39$ polynomial.
\end{theorem}

The proof of this theorem will occupy Sections~\ref{sec:model}
and~~\ref{sec:jmap} of the paper. 

We have not been able to provably determine the $\Q(\sqrt{13})$-points on the curve. The method of Chabauty does not apply in this case, and this is likely to be a difficult problem; see Section~\ref{sec:jacobians} for more about the Jacobian of $\CC$ and the difficulty of determining the $\Q$ and $\Q(\sqrt{13})$-rational points. 

We have, however, the following six points\footnote{These are all
  the points in $\CC(\Q(\sqrt{13}))$ of logarithmic height less
  than~$5.24$, according to \cite{CharliesThesis}.}
in~$\CC(\Q(\sqrt{13}))$, four of which are in~$\CC(\Q)$:

\[ \left\{  \left(1 : 3 : -2\right), \left(0 : 0 : 1\right), \left(0 : 1 : 0\right),\left(1 : 0 : 0\right), \left(3 \pm \sqrt{13} : 0 : 2\right) \right\}. \]

By evaluating the $j$-map at these points, we obtain the $j$-invariants of elliptic curves over $\Q(\sqrt{13})$ whose projective mod-$13$ image is contained in $A_4$; in fact, apart from $\left(0 : 0 : 1\right)$, whose corresponding $j$-invariant is 0, these points have projective mod-$13$ image isomorphic to $A_4$.

\begin{corollary}
\label{C2}
Elliptic curves over $\Q$ with $j$-invariant
\begin{align*}
\frac{11225615440}{1594323} &= \frac{2^{4} \cdot 5 \cdot 13^{4} \cdot 17^{3}}{3^{13}},
\\
-\frac{160855552000}{1594323} &= -\frac{2^{12} \cdot 5^{3} \cdot 11 \cdot 13^{4}}{3^{13}},
\\
\frac{90616364985637924505590372621162077487104}{197650497353702094308570556640625}
&= \frac{2^{18} \cdot 3^{3} \cdot 13^{4} \cdot 127^{3} \cdot 139^{3} \cdot
157^{3} \cdot 283^{3} \cdot 929}{5^{13} \cdot 61^{13}}
\end{align*}
have projective mod-$13$ image isomorphic to $S_4$.  Elliptic curves
over $\Q(\sqrt{13})$ with these $j$-invariants have projective mod-$13$
image isomorphic to $A_4$, as do elliptic curves over $\Q(\sqrt{13})$
with $j$-invariant \[ j =
\frac{4096000}{1594323}(15996230\pm4436419\sqrt{13}).
\]

Thus, elliptic curves over $\Q(\sqrt{13})$ with these $j$-invariants are Hasse at 13 curves over $\Q(\sqrt{13})$.
\end{corollary}

\begin{remark}
It is known that, for $l > 13$, there are no elliptic curves $E$ over $\Q$ with $H_{E,l} \cong S_4$; in fact, Serre proved that $X_{S_4}(l)(\Q)$ is empty for $l > 13$. On page 36 of \cite{Mazur3}, Mazur reports that Serre has constructed a $\Q$-point on $X_{S_4}(13)$ corresponding to elliptic curves with complex multiplication by $\sqrt{-3}$; this point that Serre found corresponds to the point $\left(0 : 0 : 1\right)$ on the curve~$\CC$ above.
\end{remark}

\begin{remark}
The rational points on $X_{S_4}(l)$ for $l \leq 11$ have already been determined. The most interesting case is $l=11$, where Ligozat proved (\cite{Lig}) that the curve $X_{S_4}(11)$ is the elliptic curve with Cremona label 121c1.
\end{remark} 

We conclude this introduction by considering the following problem, which we
would like to solve at least for every quadratic field. This may be viewed as a generalisation of Sutherland's Theorem 2 (see \ref{Suth2}). 

\begin{problem}
\label{ExceptPairs}
Fix a number field $K$. Find all exceptional pairs over $K$. 
\end{problem}

Recently, Samuele Anni has proved (see \cite{Anni}) that there can be
only finitely many exceptional primes for a given number field $K$. In
the quadratic case, his result gives the following. 

\begin{proposition}[Anni]
A quadratic field $K$ admits at most $3$ exceptional primes. If $K =
\Q(\sqrt{l})$ for $l$ a prime $\equiv 1\pmod 4$, then the only possible
exceptional primes are $7$, $11$, and $l$. If $K \neq \Q(\sqrt{l})$, then
only $7$ and $11$ are possible exceptional primes.
\end{proposition}

It is straightforward to determine, for a given quadratic field $K$,
the exceptional pairs of the form $(7,j_0)$; in principle all one needs
to do is determine the $j$-invariants of the $K$-points on the
elliptic curve 49a3.

In the case where $K = \Q(\sqrt{l})$ and the prime is $l$,
Problem~\ref{ExceptPairs} reduces to Question~\ref{Q3} above, which essentially asks for
quadratic points on the modular curves $\Xsplit(l)$; this is known to
be a difficult problem.

Regarding $11$ as a possible exceptional prime, we make the following
conjecture.

\begin{conjecture}
\label{Co1}
$11$ is not an exceptional prime for any quadratic field.
\end{conjecture}

In Section~\ref{sec:evidence}, we will explain our evidence for this
conjecture.

\subsection*{Acknowledgements} 
We are grateful to Jeroen Sijsling for helping us with the
computation of the $j$-function in Theorem~\ref{Thm2}, as well as
verifying that the genus~$3$ curves in Section~\ref{sec:jacobians} are
not isomorphic. We would like to thank Damiano Testa for making
interesting observations and suggestions regarding the curve
$X_{S_4}(13)$ and its rational and quadratic points, Tim and Vladimir
Dokchitser for their comments and suggestions regarding
Theorem~\ref{Thm1}, Alex Bartel for finding the
isomorphism~\ref{thnxalx}, Martin Orr for communicating to us results
about subvarieties of products of abelian varieties, and Andrew
Sutherland for verifying the mod-$13$ Galois images of the elliptic
curves in Corollary~\ref{C2}. Finally, we thank the anonymous referees
for their careful reading and comments, and for verifying our results.

All computations in this paper were carried out using either~\Sage
(see \cite{sage}) or \Magma (see \cite{magma}), or both.  Annotated
\Sage\ code which reproduces the computations in this paper concerning
$X_{S_4}(13)$, together with similar computations for $\Xsplit(13)$
and $\Xnonsplit(13)$, is available (see~\cite{SageCode}), as well as a
\Sage\ worksheet with the complete computation (see
\cite{SageWorksheet}).


\section{Preliminaries}\label{sec:prelim}
Let $l$ be an odd prime. We define $\PSL_2(\F_l)$ to be the kernel of
the map $ \det : \PGL_2(\F_l) \to \F_l^*/(\F_l^*)^2 \cong\{\pm1\}$. It
is isomorphic to $\SL_2(\F_l)/\left\{\pm I\right\}$. By
$\GL_2^+(\F_l)$ we mean the subgroup of matrices with square
determinant.

\begin{lemma}
\label{lemma}
Let $E/K$ be an elliptic curve. The following are equivalent. 

\begin{enumerate}
\item
$H_{E,l} \subseteq \PSL_2(\F_l);$
\item
$\sqrt{l^\ast} \in K$;
\item
$G_{E,l} \subseteq \GL_2^+(\F_l)$.
\end{enumerate}
\end{lemma}

\begin{proof}
The equivalence of (1) and (3) is clear. The equivalent of (2) and (3) follows from standard Galois theory upon observing that the determinant of $\bar{\rho}_{E,l}$ is equal to the mod-$l$ cyclotomic character over $K$. 
\end{proof}

In particular, if $E/\Q$ is an elliptic curve with $H_{E,13} \cong
S_4$, then after base-changing to $\Q(\sqrt{13})$ the projective image
is intersected with $\PSL_2(\F_{13})$, and becomes isomorphic
to~$A_4$.  This argument uses the fact that $13\equiv5\pmod8$.

We would like to briefly mention the Cartan subgroups of
$\GL_2(\F_l)$; for a complete treatment see Chapter XVIII, \S12 of
\cite{Lang}. There are two sorts of Cartan subgroup, \textbf{split}
and \textbf{non-split}. A split Cartan subgroup is conjugate to the
group of diagonal matrices, and hence is isomorphic to
$\F_l^\ast\times\F_l^\ast$. Its normaliser is then conjugate to the
group $\Csplus$ of diagonal and antidiagonal matrices. A non-split
Cartan subgroup is isomorphic to $\F_{l^2}^\ast$, and is conjugate to
the group $\Cns$ defined as follows:
\[ \Cns = \left\{\left(\begin{array}{cc}x & \delta y \\ y &
  x\end{array}\right) : x,y \in \F_l, \ (x,y) \neq (0,0)\right\},
\]
where $\delta$ is any fixed quadratic non-residue in $\F_l^\ast$. It
also has index two in its normaliser $\Cnsplus$.

Associated to both of the groups $\Csplus$ and $\Cnsplus$ are modular curves $\Xsplit(l)$ and $\Xnonsplit(l)$ respectively; these serve as coarse moduli spaces for elliptic curves $E$ whose mod-$l$ Galois image $G_{E,l}$ is contained in (a conjugate of) $\Csplus$ and $\Cnsplus$ respectively. Both curves are geometrically connected and defined over $\Q$. Over the complex numbers each curve has the description of being the quotient of the extended upper half plane $\mathcal{H}^\ast$ by an appropriate congruence subgroup. The curve $\Xsplit(l)$ is $\Q$-isomorphic to the quotient $X_0^+(l^2)$ of the modular curve $X_0(l^2)$ by the Fricke involution. Over $\C$, this isomorphism is established by mapping $\tau$ on $X_0^+(l^2)$ to $l\tau$ on $\Xsplit(l)$.

One of Sutherland's insights was that the notion of Hasse at $l$ curve $E$ over $K$ depends only on the projective mod-$l$ image $H_{E,l}$. Given a subgroup $H$ of $\PGL_2(\F_l)$, we say that $H$ is \textbf{Hasse} if its natural action on $\P^1(\F_l)$ satisfies the following two properties:
\begin{itemize}
\item
Every element $h \in H$ fixes a point in $\P^1(\F_l)$;
\item
There is no point in $\P^1(\F_l)$ fixed by the whole of $H$.
\end{itemize}

We then have the following:

\begin{proposition}[Sutherland]
\label{Hasse}
An elliptic curve $E/K$ is Hasse at $l$ if and only if $H_{E,l}$ is Hasse. 
\end{proposition}

This allows us to reduce the study of exceptional pairs largely to group theory. 

\section{Proof of Theorem \ref{Thm1}}\label{sec:pf-thm}
Throughout this proof, $K = \Q(\sqrt{5})$.

Let $E/K$ have $H_{E,5} \cong D_4$. It follows from Dickson's
classification of subgroups of $\GL_2(\F_l)$ (see \cite{Dickson}) that
$G_{E,5}$ is contained in the normaliser of a Cartan subgroup. If this
Cartan subgroup were non-split, then $G_{E,5}$ would be contained in
$\Cnsplus \cap \ \GL_2^+(\F_5)$ (we take the intersection by
Lemma~\ref{lemma}), and so $H_{E,5}$ would be contained in $(\Cnsplus
\cap \ \GL_2^+(\F_5))/\text{scalars}$, which is a group of size~$6$, and
hence cannot contain a subgroup isomorphic to $D_4$; thus $G_{E,5}
\subseteq \Csplus$, and so $E/K$ corresponds to a $K$-point on
$\Xsplit(5)$. The converse is not quite true; a $K$-point on
$\Xsplit(5)$ corresponds to an elliptic curve $E'$ over $K$ with
$H_{E',5} \subseteq D_4$, but not necessarily equal to $D_4$.

We now give an expression for the $j$-map $\Xsplit(5)
\overset{j}\longrightarrow X(1)$. Since $X_0^+(25)$ is isomorphic to $\Xsplit(5)$ under the map $\tau \mapsto 5\tau$, it suffices to write down the function $j(5\tau)$ in terms of a Hauptmodul $s$ for $X_0^+(25)$.

Let $t_N$ be a Hauptmodul for $X_0(N)$. Klein found the following
formula in 1879:
\[ j(5\tau) = \frac{(t_5^2 + 250t_5 + 3125)^3}{t_5^5}.\]
We can look up an expression for $t_5$ in terms of $t_{25}$ from \cite{Maier}:
\[ t_5 = t_{25}(t_{25}^4 + 5t_{25}^3 + 15t_{25}^2 + 25t_{25} + 25).\]
We also know that the Fricke involution $w_{25}$ maps $t_{25}$ to
$5/t_{25}$. Hence a Hauptmodul for $X_0^+(25)$ is $s := t_{25} +
5/t_{25}$. It follows that
\[ j(5\tau) = \frac{((s+5)(s^2 - 5)(s^2 + 5s+10))^3}{(s^2 + 5s + 5)^5}.\]
Inserting a $K$-value for $s$ in this expression yields the $j$-invariant of an elliptic curve $E$ over $K$ with $H_{E,5} \subseteq D_4$. The condition on $s^2 - 20$ in the statement of the Theorem ensures that we have equality here, by ensuring that the image is not contained in any one of the three subgroups of order~$2$ in~$D_4$, as we now demonstrate.

Let $E$ be a curve in $\Xsplit(5)(K)$ corresponding to a choice of $s$
in $K$, so that $H_{E,5}\subseteq D_4$. The following statements are
readily seen to be equivalent to $H_{E,5}\not= D_4$:
\begin{itemize}
\item
$H_{E,5}$ is cyclic.
\item
$G_{E,5}$ is contained in (a conjugate of) $\Cs(\F_5)$.
\item
$E$ has a pair of independent $K$-rational 5-isogenies.
\item
$E$ pulls back to a $K$-point on $X_0(25)$.
\item
$t_{25} \in K$.
\end{itemize}
Since $t_{25}$ is a root of the polynomial $x^2 - sx + 5$ of
discriminant $s^2 - 20$, we have $t_{25} \in K$ if and only if $s^2 -
20$ is a square in $K$.  Thus the statement that $s^2 - 20$ is not a
square in $K$ is equivalent to $H_{E,5}$ not being cyclic, and hence
$H_{E,5}\cong D_4$.

We have, however, overlooked an issue above. For a given $j=j(E) \in
K$ satisfying \eqref{eqn:jmapV4}, there are two other values of $s \in
K$ also satisfying \eqref{eqn:jmapV4}. This is because the field
extension $K(s)/K(j)$, which has degree~$15$ and is not Galois, has
automorphism group of order~$3$, generated by
$s\mapsto((\sqrt{5}-5)s-20)/(2s+5+\sqrt{5})$. We must ensure that for
none of the Galois conjugate values is~$s^2-20$ square in~$K$, so that
$H_{E,5}$ is not contained in any of the three cyclic subgroups
of~$D_4$.  This explains the final condition in the statement of the
Theorem.

\begin{example}
To illustrate this theorem, we input $s = 3\sqrt{5} + 1$ to obtain \[
j = \frac{337876318862280\sqrt{5} + 741305345279328}{41615795893};\]
we check that the other two values of $s \in K$, namely
$\frac{\sqrt{5} - 15}{7}$ and $\frac{-22\sqrt{5} - 30}{19}$, do not
satisfy $s^2 - 20$ is a square, and hence any elliptic curve over
$\Q(\sqrt{5})$ with this $j$ has $H_{E,5} \cong D_4$. Equivalently,
the pair $(5,j)$ is exceptional for $\Q(\sqrt{5})$.

However, if we input $s=\frac{3\sqrt{5} - 80}{41}$, we get \[j = \frac{277374956280053760\sqrt{5} + 622630488102469632}{18658757027251},\] and whilst $\frac{3\sqrt{5} - 80}{41}$ does satisfy $s^2 - 20$ not being a square, this is not the case for $s=3\sqrt{5} + 2$, which yields the same $j$-value. One therefore has to be careful of these ``pretenders'', hence the last paragraph of the above proof. 

We can even insert rational values of $s$, such as $s=1$, to obtain elliptic curves over $\Q$ whose basechange to $\Q(\sqrt{5})$ are Hasse at 5, e.g., \[ j = \frac{-56623104}{161051}.\]
\end{example}

\section{Proof of Theorem \ref{Thm2}: the model}\label{sec:model}
Let $G$ be a subgroup of $\GL_2(\Z/N\Z)$ for some $N$, and consider the modular curve $X_G(N)$ over $\Q$; let us assume $\det G = (\Z/N\Z)^\ast$, so that this curve is geometrically connected. As a curve over $\C$, the curve depends only on the intersection of $G$ with $\SL_2(\Z/N\Z)$. Therefore, if $N = 13$, and $G$ is the pullback to $\GL_2(\F_{13})$ of $S_4 \subset \PGL_2(\F_{13})$, then the modular curve $X_{S_4}(13) := X_G(13)$, when considered over $\C$, depends only on $G \cap \SL_2(\F_{13})$, which modulo scalar matrices becomes $A_4 \subset \PSL_2(\F_{13})$, and has the description $\Gamma_{A_4}(13) \backslash \mathcal{H}^\ast$, where $\Gamma_{A_4}(13)$ is the pullback of $A_4 \subset \PSL_2(\F_{13})$ to $\PSL_2(\Z)$, and $\mathcal{H}^\ast$ is the extended upper half plane.

In Chapter 3 of his thesis \cite{Gal}, Steven Galbraith describes a method to compute the canonical model of any modular curve $X(\Gamma)$, provided one can compute explicitly and to some precision the $q$-expansions of a basis of $S_2(\Gamma)$, the weight-2 cuspforms of level $\Gamma$ (a congruence subgroup). Hence, to compute the desired equation, we are reduced to computing explicitly a basis of the finite-dimensional $\C$-vector space $S_2(\Gamma_{A_4}(13))$. A standard application of the Riemann-Hurwitz genus formula gives that the genus of the desired curve is 3; this is also the dimension of $S_2(\Gamma_{A_4}(13))$. We will proceed with the exposition in a series of steps. 

\subsection{Step 1. Identifying our desired space as the set of
  invariant vectors of a representation}

Since $\Gamma(13) \subset \Gamma_{A_4}(13)$, we obtain \[ S_2(\Gamma_{A_4}(13)) \subset S_2(\Gamma(13)),\] a 3-dimensional subspace of a 50 dimensional space. Upon this latter 50 dimensional space there is a right action - the ``weight 2 slash operator'' - of $\PSL_2(\Z)$ (since $\Gamma(13)$ is normal in $\PSL_2(\Z)$) which, by definition of $S_2(\Gamma(13))$, factors through the quotient $\PSL_2(\F_{13})$, which we recall contains a unique (up to conjugacy) subgroup isomorphic to $A_4$. Our desired 3-dimensional space is then the subspace of $S_2(\Gamma(13))$ fixed by $A_4$: \[ S_2(\Gamma_{A_4}(13)) = S_2(\Gamma(13))^{A_4};\] that is, the $A_4$-invariant subspace of the $\PSL_2(\F_{13})$-representation $S_2(\Gamma(13))$.

When we carry out the computation, we will work with an explicit
subgroup of~$\PSL_2(\F_{13})$ isomorphic to~$A_4$, namely that
generated by the two matrices \[ A = \left(\begin{array}{cc}-5 & 0
  \\ 0 & 5\end{array}\right) \mbox{ and } B =
  \left(\begin{array}{cc}-2 & -2 \\ -3 & 3\end{array}\right).\] A
    different choice of $A_4$ will yield an isomorphic space of
    cuspforms, which for our application (in computing an equation for
    $X_{S_4}(13)$) makes no difference. However, the present choice of
    $A_4$ is favourable for computational reasons, since it is
    normalised by the matrix $\left(\begin{array}{cc}-1 & 0 \\ 0 &
      1\end{array}\right)$; the congruence subgroup is then said to be
      of \textbf{real type} (see~\cite[2.1.3]{Cre}).

\subsection{Step 2. The conjugate representation}

Given a congruence subgroup $\Gamma$ of level~$13$, denote by
$\widetilde{\Gamma}$ the conjugate subgroup of level~$13^2$:
\[ \widetilde\Gamma :=
\left( \begin{array}{cc} 13 & 0 \\ 0 & 1 \end{array}
\right)^{-1}\Gamma\left( \begin{array}{cc} 13 & 0 \\ 0 & 1 \end{array}
\right) \supseteq \Gamma_0(13^2)\cap\Gamma_1(13).\]  In general,
$\widetilde{\Gamma}$ has level~$13^2$; in particular we
have
\[
\widetilde{\Gamma(13)} = \Gamma_0(13^2)\cap\Gamma_1(13).
\]
Then we have the important isomorphism
\begin{align*}
S_2(\Gamma) &\longrightarrow S_2(\widetilde{\Gamma}) \\ f(z) &\longmapsto f(13z),
\end{align*}
which on $q$-expansions takes $q := e^{2\pi iz}$ to $q^{13}$. The point is that we may work with $S_2(\widetilde{\Gamma})$ instead of $S_2(\Gamma)$ if we like, as we can easily pass between the two; the two spaces are only superficially different. 

This is exactly our plan for $\Gamma(13) \subset \Gamma_{A_4}(13)$. We
have $S_2(\widetilde{\Gamma_{A_4}(13)}) \subset
S_2(\widetilde{\Gamma(13)})$. This latter space is also a
representation of $\PSL_2(\F_{13})$; for $g \in \PSL_2(\F_{13})$, we
let $\gamma$ be a pullback to $\PSL_2(\Z)$ of $g$, and define, for $F \in
S_2(\widetilde{\Gamma(13)})$, \[ g\cdot F := F |_2 \tilde{\gamma} := F |_2
\left( \begin{array}{cc} 13 & 0 \\ 0 & 1 \end{array}
\right)^{-1}\gamma\left( \begin{array}{cc} 13 & 0 \\ 0 & 1 \end{array}
\right).\] We then obtain \[ S_2(\widetilde{\Gamma_{A_4}(13)}) =
S_2(\widetilde{\Gamma(13)})^{A_4}.\] Working inside the conjugated
space $S_2(\widetilde{\Gamma(13)})$ is better, since its alternative
description as $S_2(\Gamma_0(169) \cap \Gamma_1(13))$ is more amenable
to the explicit computations we wish to carry out using the computer
algebra systems~\Sage\  and \Magma.

\subsection{Step 3. Identifying the 3 relevant sub-representations}

Inside the space $S_2(\Gamma_0(169) \cap \Gamma_1(13))$ we have the
space $S_2(\Gamma_0^+(169))$, the subspace of $w_{169}$-invariants of
$S_2(\Gamma_0(169))$. We can compute this space explicitly in
\Sage. Let $q := e^{2 \pi i z}$, $\zeta_7 := e^{2 \pi i/7}$, $\zeta_7^+
:= \zeta_7 + \zeta_7^{-1}$, and $\sigma$ a nontrivial Galois
automorphism of the field $\Q(\zeta_7^+) = \Q(\zeta_7)^+$. Then an
explicit \Sage\ computation yields \[ S_2(\Gamma_0^+(169)) = \langle
g,g^\sigma,g^{\sigma^2}\rangle,\] where \[ g(z) = q - (\zeta_7^++1)q^2
+ (1-\zeta_7^+{}^2)q^3 + (\zeta_7^+{}^2 + 2\zeta_7^+ - 1)q^4 +
\cdots.\] These three forms are Galois conjugate newforms. We will
denote by $a_n$ the Fourier coefficients of $g$.

For each $r \in \F_{13}^\ast$, define the \textbf{isotypical component
  $g_r$ of $g$} as \[ g_r := \sum_{j \equiv r \mbox{ mod }13}a_jq^j,\]
and consider the $\C$-span $V_0$ of these components. Similarly define
$V_1$ and $V_2$ by replacing $g$ with $g^{\sigma}$ and $g^{\sigma^2}$
respectively. We will show in the coming sections that each $V_i$ is a
$12$-dimensional sub-representation of $S_2(\widetilde{\Gamma(13)})$
which is irreducible as $\Q[\PSL_2(\F_{13})]$-module. We may focus on
these three sub-representations, because, as we compute later, each
one contains a unique (up to scaling) $A_4$-invariant cuspform.

Since we already know that we are looking for three forms, we need not
concern ourselves with the other irreducible components
of~$S_2(\widetilde{\Gamma(13)})$.  In fact, the sum $V_0\oplus
V_1\oplus V_2$, of dimension~$36$, is the subspace of
$S_2(\Gamma_0(169) \cap \Gamma_1(13))$ spanned by the Galois
conjugates of the newform~$g$ together with their twists by characters
of conductor~$13$.  The complementary subspace of dimension~$14$ is
spanned by oldforms from level~$13$ and their twists.  Each of these
two subspaces is the base-change of a vector space over~$\Q$ which is
irreducible as $\Q[\PSL_2(\F_{13})]$-module, while the $36$-dimensional
piece splits as $\Q(\zeta_6^+)[\PSL_2(\F_{13})]$-module into three
irreducible $12$-dimensional subspaces.

Although we discovered these facts computationally, there is an
alternative representation-theoretic explanation of these spaces in
Baran's paper \cite{Baran1}, where she shows in Propositions 3.6
and 5.2 (loc.~cit.) that the spaces $V_i$ are irreducible cuspidal
representations of $\PSL_2(\F_{13})$.

\subsection{Step 4. Computing the action of $\PSL_2(\F_{13})$ on each sub-representation}
$\PSL_2(\F_{13})$ is generated by the two matrices $S$ and $T$,
where
\[ S = \left(\begin{array}{cc}0 & -1 \\ 1 & 0\end{array}\right)
  \mbox{ and } T = \left(\begin{array}{cc}1 & 1 \\ 0 &
    1\end{array}\right).\]
However, since we have conjugated the congruence subgroup, the action
we need to consider must also be conjugated by the matrix
$\left( \begin{array}{cc} 13 & 0 \\ 0 & 1 \end{array} \right)$.
Hence, $\PSL_2(\F_{13})$ acts on~$S_2(\widetilde{\Gamma(13)})$ via the
matrices $\tilde{S}$ and $\tilde{T}$:
\[ \tilde{S} =
\frac{1}{13}\left(\begin{array}{cc}0 & -1 \\ 169 & 0\end{array}\right)
  \mbox{ and } \tilde{T} = \left(\begin{array}{cc}1 & 1/13 \\ 0 &
    1\end{array}\right).\]
Observe that the action of $\tilde{S}$ is, up to a scaling that we may
ignore, the same as the Fricke involution $w_{169}$.

Thus, to describe the action of $\PSL_2(\F_{13})$ on each $V_i$, we will express the action of $\tilde{S}$ and $\tilde{T}$ on each $V_i$, explicitly as $12 \times 12$ matrices. 

\subsection{Step 5. Computing the action of $\tilde{S}$ and $\tilde{T}$}
We fix $i=0$; the other two cases are completely analogous and can be
obtained by Galois conjugation (see Lemma~\ref{lem:commutes} below).

To compute the action of $\tilde{T}$ on $V_0$, we use the definition
directly:
\[ \left(g|_2\left(\begin{array}{cc}1 & 1/13 \\ 0 & 1\end{array}\right)\right)(z) = g(z + \frac{1}{13}).\]
Recall that $a_i$ is the $i$th coefficient of $g$. We then get \[g(z + \frac{1}{13}) = \zeta_{13}q - (s+1)\zeta_{13}^2q^2 + \cdots\] which we can rearrange as \[\zeta_{13}(a_1q + a_{14}q^{14} + a_{27}q^{27} + \cdots) + \zeta_{13}^2(a_2q^2 + a_{15}q^{15} + \cdots) + \cdots.\]

Thus, in the isotypical basis for~$V_0$, the action of $\tilde{T}$ is
given simply by the $12 \times 12$ diagonal matrix \[ \left(\begin{array}{cccc}
  \zeta & & & \\ & \zeta^2 & & \\ & & \ddots & \\ & & & \zeta^{12}
\end{array}\right)\]
where we write $\zeta$ for $\zeta_{13}$. In particular, this shows
that $V_0$ is indeed invariant under the action of $\tilde{T}$.

Computing $\tilde{S}$ directly on the isotypical basis is not so easy,
so what we do is change to a basis upon which we can compute
it. Instead of the isotypical basis, we take the \textbf{twist
  basis} \[ \langle g \otimes \chi^j : 0 \leq j \leq 11\rangle,\]
where $\chi : 2 \mapsto \zeta_{12}$ is a fixed generator of the group
of Dirichlet characters of conductor~$13$, and $g \otimes \chi$
denotes the usual twist of~$g$ by~$\chi$. Note that this twist basis
consists entirely of newforms (see \cite{Atkin-Li}). Since twisting
by~$\chi$ preserves~$V_0$ and the change of basis matrix is
$(\chi^j(i))$ (for $0 \leq j \leq 11$ and $1 \leq i \leq 12$), which has
nonzero determinant, we have shown the following.

\begin{lemma}
Both the isotypical and twist bases are $\C$-bases for the $12$-dimensional
subspace~$V_0$ of~$S_2(\widetilde{\Gamma(13))})$:
\[\langle g \otimes \chi^j : 0 \leq j \leq 11\rangle = \langle g_j : 1 \leq j \leq 12\rangle.\]
\end{lemma}

Recall that the action of $\tilde{S}$ is the same as the Fricke
involution $w_{169}$. It is known (see \cite{Atkin-Li}) that $w_N$ acts
on newforms $F$ of level $N$ as \[ F |_2 w_N =
\lambda_N(F)\cdot\bar{F},\] where $\bar{F}$ is the newform obtained
from the Fourier expansion of $F$ by complex conjugation, and
   $\lambda_N(F)$ is the {Atkin-Lehner pseudoeigenvalue}, an
algebraic number of absolute value 1 (Theorem 1.1 of
\cite{Atkin-Li}). In our twist basis, we have \[ \overline{g \otimes
  \chi^j} = g \otimes \chi^{12 - j},\] so we only need to compute the
pseudoeigenvalues associated to $g \otimes \chi^j$ for $0 \leq j \leq
6$;  the others may be obtained from these by complex
conjugation. Also, the pseudoeigenvalues for $j=0$ and $j=6$ are
actually eigenvalues, and may be computed directly (for example in
\Sage); we find that the eigenvalue for $j=0$ is~$+1$, and for $j=6$
is~$-1$. 

\subsection{Step 6. Computing the Atkin-Lehner pseudoeigenvalues}

In order to stay consistent with the notation of \cite{Atkin-Li}, we relabel $g$ to $F$, and we let $q=13$. By $a(q)$ we mean the $q$th Fourier coefficient of $F$, which we may check is 0. We may also check that $F$ is not a twist of an oldform of $S_2(\widetilde{\Gamma(13)})$; thus, in the language of \cite{Atkin-Li}, $F$ is \textbf{13-primitive}. We let $\chi_0$ be the trivial character modulo~$13$, so $\chi_0 = \chi^0$, and we write $\lambda(\chi)$ for the Atkin-Lehner pseudoeigenvalue of $F \otimes \chi$, for $\chi$ any character. We let $g(\chi)$ be the Gauss sum of the character $\chi$, with the convention that $g(\chi_0) = -1$. 

The main tool to compute $\lambda(\chi^j)$, for $0 \leq j \leq 11$, is Theorem 4.5 in \cite{Atkin-Li}, which in the present context is as follows:

\begin{theorem}[Special case of Theorem 4.5 of \cite{Atkin-Li}]
With the above notation and assumptions, we have, for $0 \leq j \leq 11$,
\[ (-1)^j12g(\chi^{12-j})\lambda(\chi^j) = \sum_{k=0}^{11}g(\chi^k)g(\chi^{j+k})\overline{\lambda(\chi^k)}.\]
\end{theorem}

This theorem gives us, for each $0 \leq j \leq 11$, a linear relation
among the $\lambda(\chi^k)$. Although there are twelve
$\lambda(\chi^k)$, we have in the previous paragraph computed two of
them, leaving us with 10. But actually, we have $\lambda(\chi^j) =
\overline{\lambda(\chi^{12-j})}$ for $0 \leq j \leq 5$, so we really
only have 5 independent unknowns. However, our strategy is, at first,
to consider that we indeed have 10 unknowns (namely, $\lambda(\chi^j)$
for $1 \leq j \leq 5$ and $7 \leq j \leq 11$) and use the theorem to
derive as many linear relations between these 10 unknowns as we can.

Doing this yields 6 independent equations, whose coefficients lie in
$\Q(\zeta_{156})$ (the field over which the Gauss sums are
defined). One is, however, able to obtain two more independent
equations, by applying Theorem 4.5 of Atkin and Li starting not with
$F=g$ (as we did previously), but rather with $F = g \otimes
\chi^6$. Thus we get

\begin{theorem}[Another special case of Theorem 4.5 of \cite{Atkin-Li}]
For $0 \leq j \leq 11$, we get
\[ (-1)^{j+1}12g(\chi^{12-j})\lambda(\chi^{6+j}) = \sum_{k=0}^{11}g(\chi^k)g(\chi^{j+k})\overline{\lambda(\chi^{6+k})}.\]
\end{theorem}

As previously stated, this yields two more independent equations,
giving us a linear system of 8 independent equations in 10 unknowns.

Let $x = \lambda(\chi)$ and $y = \lambda(\chi^2)$. We obtain the
following two linear equations in the unknowns
$x,\overline{x},y,\overline{y}$: 
\begin{align}
\label{one}c_1\bar{y} + c_2 y + c_3 x + c_4\bar{x} &= c_5 \\ \label{two} c_6y + c_7x + c_8\bar{x} &= c_9;
\end{align}
here the $c_i$ are explicit elements of $\Q(\zeta_{156})$. We now use
the relations $x\bar{x} = y\bar{y} = 1$. We use \eqref{two} to
eliminate $y$ and $\bar{y}$ from \eqref{one} to obtain a linear
relation between $x$ and $\bar{x}$; now using $x\bar{x}=1$, we obtain
a quadratic in $x$. This quadratic has no root in $\Q(\zeta_{156})$;
we need to adjoin $\sqrt{-7}$, so in fact we work in the field
$\Q(\zeta_{1092})$; this might seem excessive, but the coefficients of
$g$ are anyway in $\Q(\zeta_7)^+$. This quadratic in $x$ tells us that
$x$ is one of two values, and $x$ determines all other
$\lambda(\chi^j)$.

In order to determine which of the two values $x$ really is, we computed two competing $\tilde{S}$ matrices, and took the one which satisfied the correct relations with $\tilde{T}$ to be the generators of $\PSL_2(\F_{13})$, namely, \[ \tilde{S}^2 = \tilde{T}^{13} = (\tilde{S}\tilde{T})^3 = 1.\]

\subsection{Step 7. The cuspforms}
We now have matrices giving the action of $\tilde{S}$ on the twist
basis, and the action of $\tilde{T}$ on the isotypical basis; a change
of basis matrix applied to either of these gives the action of both
matrices in terms of the same basis. Write $\rho(S)$ and $\rho(T)$ for
the $12 \times 12$ matrices giving the action of $\tilde{S}$ and
$\tilde{T}$ respectively with respect to the twist basis.

We now compute the $A_4$-invariant subspace of $V_0$. Recall that our
generators of $A_4 \subset \PSL_2(\F_{13})$ are: 
\[ A = \left(\begin{array}{cc}-5 & 0 \\ 0 & 5\end{array}\right) \mbox{
    and } B = \left(\begin{array}{cc}-2 & -2 \\ -3 &
    3\end{array}\right).\]
Writing each generator as a word in $S$ and $T$:
\begin{align*}
A &= T^5ST^{-2}ST^2ST^3ST^{-5},\\ B &= T^4ST^3ST^{-3}S.
\end{align*}
the action of $A_4$ on $S_2(\widetilde{\Gamma(13)})$ is given by the
same words in the matrices $\tilde{S},\tilde{T}$:
\begin{align*}
\tilde{A} &= \tilde{T}^5\tilde{S}\tilde{T}^{-2}\tilde{S}\tilde{T}^2\tilde{S}\tilde{T}^3\tilde{S}\tilde{T}^{-5},\\ \tilde{B} &= \tilde{T}^4\tilde{S}\tilde{T}^3\tilde{S}\tilde{T}^{-3}\tilde{S}.
\end{align*}
The action of $\tilde{A}$ and $\tilde{B}$ on our vector space
$V_0$ is given by taking the same words as above, but in $\rho(S)$ and
$\rho(T)$; we call the resulting matrices $\rho(A)$ and $\rho(B)$.

The intersection of the kernels of $\rho(A)-I$ and
$\rho(B)-I$ is one-dimensional, spanned by a vector of the
coefficients, in the twist basis, of an $A_4$-invariant cuspform in
$V_0$. These coefficients lie in the degree~$9$ field
$\Q(\zeta_7^+,\zeta_{13}^{++})$, where by $\Q(\zeta_{13}^{++})$ we
denote the unique cubic subfield of $\Q(\zeta_{13})$.  We call this
$A_4$-invariant form $f$.

We do not have to repeat the calculation for $V_1$ and $V_2$,
because of the following fact. Here we regard $V_i$ as
$\bar{\Q}[\PSL_2(\F_{13})]$-modules.

\begin{lemma}\label{lem:commutes}
Let $\gamma$ be an element of $\PSL_2(\F_{13})$. The following diagram commutes:
\begin{center}
{
$\minCDarrowwidth40pt
\begin{CD}
V_0 @>\sigma>> V_1\\
@V\gamma VV  @VV\gamma V\\
V_0 @>\sigma>> V_1
\end{CD}$
}
\end{center}
\end{lemma}

\begin{proof}
Each $V_i$ admits a twist basis, corresponding to $g^{\sigma^i}$ and its twists under powers of $\chi$. Fixing this twist basis for each $V_i$, we find that the action of $\tilde{S}$ and $\tilde{T}$ is exactly the same; this is because the coefficients in $\tilde{S}$ and $\tilde{T}$ we found for $V_0$ are invariant under the action of $\sigma$. 
\end{proof}


The lemma allows us to conclude that for $i=0,1,2$, the conjugate
$f^{\sigma^i}$ spans the $A_4$-invariant subspace of~$V_i$, and hence
that $\left\{f,f^\sigma,f^{\sigma^2}\right\}$ is a basis of
$S_2(\widetilde{\Gamma_{A_4}(13)})$.   Next we replace this
basis with one defined over a smaller field, namely
$\Q(\zeta_{13}^{++})$.

Write $f$ as \[ f = F + \zeta_7^+G + \zeta_7^{+2}H,\] where $F,G,H$
have coefficients in $\Q(\zeta_{13}^{++})$.  The forms $F,G,H$ form a
basis for the same space, with coefficients in the smaller field:

\begin{lemma}
The following two $\mathbb{C}$-spans are the same: \[ \langle f,f^\sigma,f^{\sigma^2}\rangle = \langle F,G,H\rangle.\]
\end{lemma}

\begin{proof}
We have
\[\left(\begin{array}{c}f \\ f^\sigma
  \\ f^{\sigma^2}\end{array}\right) = \left(\begin{array}{ccc}1 &
  \zeta_7^+ & \zeta_7^{+2} \\1 & \sigma(\zeta_7^+) &
  \sigma(\zeta_7^{+2}) \\1 & \sigma^2(\zeta_7^{+}) &
  \sigma^2(\zeta_7^{+2})\end{array}\right)\left(\begin{array}{c}F \\ G
  \\ H\end{array}\right)\]
where the matrix has nonzero determinant.
\end{proof}

As a final flourish, we apply the following nonsingular
transformation \[\left(\begin{array}{ccc}1 & 4 & 3 \\-4 & -3 & 1 \\6 &
  -2 & 5\end{array}\right)\] to obtain the following cuspforms (where
  again $\zeta = \zeta_{13}$) which are a basis for $S_2(\widetilde{\Gamma_{A_4}(13)})$:
{\small
\begin{align*}
f =\ &-q + (-\zeta^{11} - \zeta^{10} - \zeta^3 - \zeta^2)q^2 + (\zeta^{11} + \zeta^{10} - \zeta^9 - \zeta^7 - \zeta^6 - \zeta^4 + \zeta^3 + \zeta^2 - 2)q^3 + \cdots \\
g =\ &(-\zeta^{11} - \zeta^{10} - \zeta^9 - \zeta^7 - \zeta^6 - \zeta^4 - \zeta^3 - \zeta^2 - 1)q + (-\zeta^{11} - \zeta^{10} - \zeta^9 - \zeta^7 - \zeta^6 - \zeta^4 - \zeta^3 - \zeta^2 - 2)q^2 + \\ &(-\zeta^{11} - \zeta^{10} - \zeta^3 - \zeta^2 - 1)q^3 + \cdots \\
h =\ &(\zeta^{11} + \zeta^{10} + \zeta^3 + \zeta^2 + 3)q + (-\zeta^{11} - \zeta^{10} - \zeta^9 - \zeta^7 - \zeta^6 - \zeta^4 - \zeta^3 - \zeta^2 - 3)q^2 + q^3 + \cdots.
\end{align*}
} The final transformation was chosen retrospectively, solely for
cosmetic reasons; it moves three of the rational points on the curve
to $[1:0:0], [0:1:0], [0:0:1]$.

Having obtained the $q$-expansions, we may proceed with the canonical
embedding algorithm of Galbraith, to obtain the smooth quartic
equation for the model~$\CC$ given in the introduction.  In practice
this simply amounts to finding a linear relation between
the~$q$-expansions of the $15$ monomials of degree~$4$ in $f,g,h$.
Although these $q$-expansions have coefficients defined over a cubic
field (and there is no basis with rational $q$-expansions), the
relation we find has rational coefficients.

\begin{remark}
In her paper \cite{Baran1}, Burcu Baran uses a different method to
compute the equation of the modular curve $\Xnonsplit(13)$; her method
would also work for the present curve $X_{S_4}(13)$; one would need an
analogue of her Proposition 6.1 for the subgroup at hand, which can be
proved using her formulae in \S 3 of loc.cit.
\end{remark}

\begin{remark}
We also implemented a variation of the approach detailed here, using a
modular symbol space of level $169$, dual to the spaces $V_i$ above.
This second approach saved us from having to find the
pseudoeigenvalues, since the matrices of both $S$ and~$T$ on modular
symbols are easily computed.  This variation is also easy to adapt to
find models for the curves $\Xsplit(13)$ and~$\Xnonsplit(13)$.  Full
details (including the cases $\Xsplit(13)$ and $\Xnonsplit(13)$) may
be found in the annotated \Sage\ code \cite{SageCode} and \Sage\ worksheet
\cite{SageWorksheet}.

\end{remark}

\section{Proof of Theorem \ref{Thm2}: the $j$-map}\label{sec:jmap}
In this section we explicitly determine the $j$-map \[ X_{S_4}(13)
\overset{j}\longrightarrow X(1) \cong \P^1_\Q\] as a rational function
on $X_{S_4}(13)$. This is a function of degree~$91$, which we seek to
express in the form \[ j(X,Y,Z) = \frac{n(X,Y,Z)}{d_0(X,Y,Z)},\] where
$n$ and $d_0$ are polynomials of the same degree over $\Q$.  We first
find a suitable denominator~$d_0(X,Y,Z)$.  The poles of~$j$ are all of
order~$13$ and are at the $7$ cusps of $X_{S_4}(13)$, so we will find
these, as $\overline{\Q}$-rational points on~$X_{S_4}(13)$.  Then we
find a cubic~$d$ in $\Q[X,Y,Z]$ which passes through these $7$ points
(there is no quadratic which does), and set ~$d_0=d^{13}$.  Having
found~$d_0$ we determine the numerator~$n$ using linear algebra on
$q$-expansions.

\begin{remark}
It would also be possible, in principal, to follow \cite{Baran1} by
computing the zeros of~$j$ numerically to sufficient precision to be
able to recognise them as algebraic points, as then we would have the
full divisor of the function~$j$ from which $j$ itself could be
recovered using an explicit Riemann-Roch space computation.  Our
method has the advantage of not requiring any numerical
approximations.
\end{remark}

We first need to find which points on our model~$\CC$
for~$X_{S_4}(13)$ are the $7$ cusps.  It turns out that there are three
which are defined and conjugate over the degree~$3$ subfield
$\Q(\alpha)$ of $\Q(\zeta)$, where $\zeta=\zeta_{13}$ and
$\alpha=\zeta+\zeta^5+\zeta^8+\zeta^{12}$, and the other four are
defined and conjugate over the degree~$4$ subfield $\Q(\beta)$ of
$\Q(\zeta)$, where $\beta=\zeta+\zeta^3+\zeta^9$.

\begin{proposition}\label{prop:cusps}
On the model~$\CC$ for $X_{S_4}(13)$, the $7$ cusps are given by
the three Galois conjugates of
\[
   [-3\alpha^2-7\alpha+1 : 4\alpha^2+11\alpha-3 : 5]
\]
and the four conjugates of
\[
  [3\beta^3+6\beta^2+6\beta-15 : \beta^3+\beta^2-4\beta-4 : 9],
\]
where $\alpha$ and~$\beta$ have minimal polynomials
$x^3+x^2-4x+1$ and $x^4 + x^3 + 2x^2 - 4 x + 3$ respectively.
\end{proposition}

The degree 3 cusps are easy to obtain; the cusp corresponding to the
point $i\infty$ on the extended upper half-plane~$\mathcal{H}^\ast$
has coordinates given by the leading coefficients of the three basis
cuspforms $f,g,h$; denoting by $\varphi$ the map
\begin{align*}
\varphi : \Gamma_{A_4}(13) \backslash \mathcal{H}^\ast
&\overset\sim\longrightarrow X_{S_4}(13) \\ \Gamma_{A_4}(13)\cdot z
&\longmapsto[f(z) : g(z) : h(z)],
\end{align*}
we see that $\varphi(i\infty) = [a_1(f) : a_1(g) : a_1(h)]$.
Expressing these coordinates in terms of $\alpha$ gives the degree $3$
cusp given in the proposition.

It is possible to determine in advance the Galois action on the cusps,
as in the following Lemma.  However, note that in practice our method
to compute the cusps algebraically, given below, does not require this
knowledge in advance.

\begin{lemma}
The absolute Galois group of~$\Q$ acts on the seven cusps with two
orbits, of sizes $3$ and~$4$.
\end{lemma}

\begin{proof}
We know \textit{a priori} that the cusps are all defined over
$\Q(\zeta_{13})$.  Theorem 1.3.1 in \cite{Stevens} explains how to
compute the action of $\Gal(\Q(\zeta_{N})/\Q)\cong (\Z/N\Z)^*$ on the
cusps of a modular curve~$X$ of level~$N$, provided that the field of
rational functions on~$X$ is generated by rational functions whose
$q$-expansions have rational coefficients.  This does not apply here,
since the field of modular functions for $\Gamma_{A_4}(13)$ is not
generated by functions with rational $q$-expansions, but rather by
functions with $q$-expansions in the cubic field~$\Q(\alpha)$.  But
following Stevens' method we can compute the action of the absolute
Galois group of $\Q(\alpha)$, which acts through the cyclic subgroup
of order~$4$ of $(\Z/13\Z)^*$ fixing~$\alpha$.  We find that it fixes
three cusps (which we already know from above, as they are defined
over~$\Q(\alpha)$), and permutes the remaining four cyclically.  It
follows that the other four cusps are also permuted cyclically by the
full Galois group, and hence have degree~$4$ as claimed.
\end{proof}

It remains to find the coordinates of one cusp of degree~$4$.

Let $c \in \Gamma_{A_4}(13) \backslash \P^1(\Q)$ be any cusp.  Then
there exists $\gamma \in \PSL_2(\Z) \backslash \Gamma_{A_4}(13)$ such
that $\gamma(c) = \infty$, and hence, \[\alpha(c) = [a_1(f|\gamma) :
  a_1(g|\gamma) : a_1(h|\gamma)].\] Since we already computed in the
previous section the action of $\PSL_2(\Z)$ on the cuspforms $f,g,h$,
we can compute the right-hand side of this equation for any $\gamma$.
With some work one can show that the cubic cusps are obtained using
$c=\infty, 1$ and~$7/6$, while the quartic cusps are obtained from
$c=2,3,6$ and~$9$; or we can simply choose random $\gamma \in
\PSL_2(\Z)$ until we find a point which is not one of the three
conjugates we already have. This proves Proposition~\ref{prop:cusps}.

Next we find a cubic curve passing through these 7 points.

\begin{proposition}
The following cubic passes through the seven cusps:
\[
5 X^{3} - 19 X^{2} Y - 6 X Y^{2} + 9 Y^{3} + X^{2} Z - 23 X Y Z - 16
Y^{2} Z + 8 X Z^{2} - 22 Y Z^{2} + 3 Z^{3}.
\]
\end{proposition}

\begin{proof}
The full linear system of degree~$3$ associated
to~$\mathcal{O}_{\P^2}(1)$ has dimension~$10$, and the subsystem
passing though the $7$ cusps has dimension~$3$ with a basis in
$\Q[X,Y,Z]$.  Using LLL-reduction we found a short element which does
not pass through any rational points on~$\CC$ (to simplify the
evaluation of the $j$-map at these points later).
\end{proof}

Since all cusps have ramification degree 13 under the $j$-map, a
possible choice for the denominator of the $j$-map is to
take the 13th power~$d_0=d^{13}$ of this cubic.

Next we turn to the numerator~$n(X,Y,Z)$, which is a polynomial of
degree~$39$.  The idea is to consider an arbitrary degree~$39$
polynomial in the $q$-expansions of the cusp forms $f,g,h$, and
compare it with the known $q$-expansion of $j\cdot d(f,g,h)^{13}$.  This
gives a system of linear equations which can be solved.

The full linear system of degree~$39$ has degree~$820$, but modulo the
defining quartic polynomial for~$\CC$ we can reduce the number of
monomials needing to be considered to only~$154$.  We chose those
monomials in which either $X$ does not occur, or else $Y$ does not
occur and $X$ has exponent~$1$ or~$2$, but this is arbitrary.

Now we consider the equation \[ n(X,Y,Z) - j(X,Y,Z)\cdot d(X,Y,Z)^{13} =
0,\] as a $q$-expansion identity after substituting $f,g,h$ for
$X,Y,Z$.  Using $250$ terms in the $q$-expansions (giving a margin to
safeguard against error) and comparing coefficients gives $250$
equations for the unknown coefficients of $n(X,Y,Z)$.  There is a
unique solution, which has rational coefficients.  Although we have
apparently only shown that the equation holds modulo~$q^{250}$, it
must hold identically, since we know that there is exactly one
solution.

The expression for $n(X,Y,Z)$ we obtain this way is too large to
display here (it has $151$ nonzero integral coefficients of
between~$46$ and $75$ digits), but can easily be used to evaluate the
$j$-map at any given point on the curve~$\CC$.  For the sake of
completeness, however, we give here explicitly the zeroes of the
$j$-map from which (together with the poles) it may be recovered; the
complete expression may be seen in \cite{SageWorksheet}.

The $91$ zeroes of $j$ consist of $29$ points with multiplicity~$3$
and $4$ with multiplicity~$1$, all defined over the number field
$M=\Q(\delta)$, where $\delta$ has minimal polynomial
\[
x^8 - 9x^6 + 32x^4 - 9x^2 + 1,
\]
which is Galois with group~$D_4$.  This field is the splitting field
of the polynomial~$P(t)$ defined in the next section, so is also the
field of definition of the points in the fibre over~$j=0$ of the
covering map $X_0(13)\to X(1)$.
Some of the $33$ zeroes are defined over the quartic subfields
$\Q(\alpha)$ and $\Q(\beta)$, where $\alpha$ and~$\beta$ have minimal
polynomials~$x^4 + 13x^2 - 39$ and~$x^4 - 13x^2 + 52$ respectively.
Their coordinates are as follows (together with all Galois
conjugates);  with multiplicity~$1$ we have
\[
[3\beta^3 + 2\beta^2 - 15\beta - 14 :
-3\beta^3 + 4\beta^2 + 29\beta - 22 :
-3\beta^3 - 4\beta^2 + 25\beta + 46 ],
\]
and with multiplicity~$3$ we have the rational point~$[1:0:0]$, the
degree~$4$ points
\begin{gather*}
[2(-\alpha^2 + 5\alpha - 4) : 
-\alpha^3 - \alpha^2 + \alpha + 6 :
\alpha^3 + 14\alpha - 35],\\
[\beta^3 - 2\beta^2 - 9\beta + 14 :
2(\beta^2 - \beta - 2) :
2(-\beta^3 - 2\beta^2 + 7\beta + 16) ],\\
[4(\beta - 1) :
\beta^3 - 7\beta - 10 :
2(\beta^3 + 2\beta^2 - 7\beta - 12) ],
\end{gather*}
and the degree~$8$ points
\begin{multline*}
[\delta^7 + 2\delta^6 - 8\delta^5 - 8\delta^4 +  36\delta^3 +
  12\delta^2 - 5\delta - 2 :
8(\delta^3 + \delta^2) :\\
-\delta^7 - 2\delta^6 + 4\delta^5 - 28\delta^3 - 8\delta^2 + 5\delta +
2]
\end{multline*}
and 
\begin{multline*}
[2(-\delta^6 + 4\delta^5 - 10\delta^3 + 4\delta - 1) :
-3\delta^7 + 2\delta^6 + 28\delta^5 - 32\delta^4 - 52\delta^3 +
16\delta^2 + 7\delta - 2 :\\
\delta^7 - 2\delta^6 - 4\delta^5 + 4\delta^4 + 4\delta^3 + 8\delta^2 -
9\delta + 2 ]. 
\end{multline*}

\section{Proof of Corollary \ref{C2}}\label{sec:pf-cor}
Evaluating the $j$-map at the six points in $\CC(\Q(\sqrt{13}))$
exhibited in the introduction yields the five $j$-invariants listed in
the statement of Corollary~\ref{C2}, together with $j=0$, which is the
image of~$[1:0:0]$. We know that any elliptic curve~$E$ over
$\Q(\sqrt{13})$ with one of these six $j$-invariants has $H_{E,13}
\subseteq A_4$.  Any elliptic curve with $j=0$ has complex
multiplication, with mod-$13$ image contained in a split Cartan
subgroup (split since $13\equiv1\pmod3$).  Hence what remains to prove
in this section is that $H_{E,13} \cong A_4$ for the five nonzero
$j$-invariants listed.

\begin{lemma}
Let $l$ be a prime for which $X_0(l)$ has genus~$0$ (that is, $l = 2,3,5,7,13$). There is an explicit polynomial $F_l(X,Y) \in \Z[X,Y]$ such that, if $E/K$ is an elliptic curve over a number field, then \[ H_{E,l} \cong \Gal(F_l(X,j(E))).\]
\end{lemma}

\begin{proof}
The function field of $X_0(l)$ is generated by a single modular
function~$t$ (the so-called ``Hauptmodul''), and classically there is
a canonical choice of such, for each $l$.  The $j$-function is a
rational function of $t$ of degree $l+1$ of the form $P(t)/t$, where
$P$ is an explicit integral polynomial of degree $l+1$.

Define $F_l(X,Y) = P(X) - YX \in \Z[X,Y]$. Let $E/K$ be an elliptic
curve over a number field, and consider the set of roots of
$F_l(X,j(E)) \in K[X]$ over $\bar{\Q}$. As a set, this is in bijection
with the set of preimages~$t$ of~$j(E)$ under the $j$-map $X_0(l)\to
X(1)$ (which is unramified away from $j=0$ and $j=1728$), and hence is
in Galois equivariant bijection with the $l$-isogenies on $E$.  Hence
the Galois action on the set of $l+1$ isogenies is isomorphic to the
Galois action on the roots of $F_l(X,j(E))$.
\end{proof}


For $l=13$, we have $P(t)=(t^{2} + 5t + 13) \cdot (t^{4} +
7t^{3} + 20t^{2} + 19t + 1)^{3}$, and hence
\[
 F_{13}(X,Y) = (X^{2} + 5 X + 13) \cdot (X^{4} + 7 X^{3} + 20 X^{2} +
 19 X + 1)^{3} - XY.
\]

For each $j$-invariant listed in Corollary~\ref{C2} we may verify that
$F_{13}(X,j)$ has Galois group isomorphic to~$A_4$
over~$\Q(\sqrt{13})$, and for the rational $j$-values, isomorphic
to~$S_4$ over~$\Q$.


\section{Proof of Proposition \ref{P3}}\label{sec:pf-gp}
By Proposition~\ref{Hasse} and Lemma~\ref{lemma}, Proposition~\ref{P3} is equivalent to the following purely group theoretic statement.

\begin{proposition}\label{prop:ell-congruence}
Let $H \subseteq \PSL_2(\F_l)$. Then $H$ is Hasse if and only if one of the following holds.
\begin{enumerate}
\item
$H \cong A_4$ and $l \equiv 1\pmod{12}$;
\item
$H \cong S_4$ and $l \equiv 1\pmod{24}$;
\item
$H \cong A_5$ and $l \equiv 1\pmod{60}$;
\item
$H \cong D_{2n}$ and $l \equiv 1\pmod{4}$, where $n > 1$ is a divisor of $\frac{l-1}{2}$, and the pullback of $H$ to $\GL_2(\F_l)$ is contained in the normaliser of a split Cartan subgroup.
\end{enumerate}
\end{proposition}


We begin the forward implication of this Proposition by quoting the
following lemma of Sutherland, which is a small piece of his Lemma 1.

\begin{lemma}[Sutherland]
If $H \subseteq \PSL_2(\F_l)$ is Hasse, then $l \nmid |H|$.
\end{lemma}

We may now invoke the following classical result (see \cite[Theorem
  XI.2.3]{LangModForms}).

\begin{fact}
\label{classical}
Let $H$ be a subgroup of $\PGL_2(\F_l)$ with $l \nmid |H|$, and let $G$ denote its pullback to $\GL_2(\F_l)$. Then one of the following occurs:

\begin{itemize}
\item
$H$ is cyclic, and $G$ is contained in a Cartan subgroup;
\item
$H$ is dihedral, and $G$ is contained in the normaliser of a Cartan subgroup;
\item
$H$ is isomorphic to $A_4$, $S_4$ or $A_5$.
\end{itemize}
\end{fact}

Clearly $H$ being cyclic is incompatible with $H$ being Hasse, so either $H \cong D_{2n}$ for $n > 1$, or $H$ is one of $A_4$, $S_4$ or $A_5$. 

\begin{lemma}
\label{divides}
Let $H \subseteq \PSL_2(\F_l)$ be Hasse, and let $h \in H$. Then the order of $h$ must divide $\frac{l-1}{2}$. 
\end{lemma}

\begin{proof}
Write $H' := \langle h\rangle$, a cyclic group of order $r$ say, prime
to $l$. By Fact~\ref{classical}, the pullback $G'$ of $H'$ to
$\GL_2(\F_l)$ is contained in a Cartan subgroup. If this Cartan
subgroup were non-split, then the elements of $G'$ would \emph{not} be
diagonalisable, and hence $h$ would not fix an element of
$\P^1(\F_l)$, contradicting the Hasse assumption. Thus the Cartan
subgroup must be split, so the elements of $G'$ are diagonalisable,
and thus $h$ has two fixed points; the same is true for every
non-identity element of $H'$. We now apply the orbit counting lemma to
$H'$:
\begin{equation}
\label{orbit}
s := |\P^1(\F_l)/H| = 2 + \frac{l-1}{r}.
\end{equation} 

Note that this formula says that there are $\frac{l-1}{r}$ non-trivial orbits of $\P^1(\F_l)$ under $h$ (a trivial orbit being a fixed point). The sizes of these $\frac{l-1}{r}$ non-trivial orbits all divide $r$ and sum to $l-1$, and hence they are all equal to $r$. 

We claim that $s$ must be even. This is clear if $r$ is odd, by \eqref{orbit}. If $r$ is even, then \[\mbox{sign}(h) = (-1)^{s-2} = (-1)^s,\] where sign means the sign as a permutation. The key observation, which proves the claim, is that sign$(h)$ must be 1, because it coincides with $\det h$. 

As $s$ must be even, we look finally at \eqref{orbit} to conclude that $r$ must divide $l-1$ with even remainder, and the lemma is proved.
\end{proof}

A part of the previous proof is worth framing, for it explains why the pullback of the dihedral group $D_{2n}$ is contained in the normaliser of a \emph{split} Cartan subgroup.

\begin{lemma}
Let $H \subseteq \PSL_2(\F_l)$ be Hasse, and let $h \in H$. Let $H' := \langle h\rangle$, and let $G'$ be the pullback of $H'$ to $\GL_2(\F_l)$. Then $G'$ is contained in a {split} Cartan subgroup.
\end{lemma}

Lemma~\ref{divides} implies that the $n$ in $D_{2n}$ divides
$\frac{l-1}{2}$, and also the congruence restrictions for $A_4$, $S_4$
and $A_5$; indeed, since $A_4$ contains elements of order 1, 2 and 3,
we must have that 2 and 3 divide $\frac{l-1}{2}$, or equivalently, $l
\equiv 1\pmod{12}$; the same argument works for $S_4$ and $A_5$. This
proves the forward implication of the group theoretic proposition
above.

We now prove the converse, that is, if $H$ is isomorphic to one of the four subgroups listed above, then it satisfies the Hasse condition. 

The easier thing to prove is that every element $h$ in $H$ fixes a
point of $\P^1(\F_l)$, so we address this first. Suppose, for a
contradiction, that we have $h \in H$ which fixes no point of
$\P^1(\F_l)$, let $r$ be the order of this $h$, and let $s := |\P^1(\F_l)/h|$ be the number of orbits. Proposition 2 of \cite{Drew} says that sign($h$) = $(-1)^s$; in particular, $s$ must be even. Applying the orbit counting lemma to the action of $\langle h\rangle$ on $\P^1(\F_l)$ yields the formula $s = (l+1)/r$, and hence $r$ must divide $\frac{l+1}{2}$. We now do a case-by-case elimination. Suppose first that $H \cong D_{2n}$ with all the other conditions expressed above. The order of any element of this group must divide $\frac{l-1}{2}$. Since no numbers divide both $\frac{l-1}{2}$ and $\frac{l+1}{2}$, we obtain the desired contradiction. The argument in the other cases is similar.

We are left with proving that, in the four cases, no point of
$\P^1(\F_l)$ is fixed by the whole of $H$. This follows from the
following well-known fact from group theory; see for example Theorem 80.27 in \cite{CurtisReiner}.

\begin{lemma}
\label{Mackey}
Let $G$ be a group, $S$ a transitive left $G$-set, and $H$ a subgroup of $G$. Denote by $H\backslash S$ the set of orbits of $S$ under $H$. Let $B$ denote the $G$-stabiliser of any point of $S$. Then we have an isomorphism of $H$-sets
\[ S \cong \bigsqcup_{g}H/(H\cap B^g),\]
where $g$ runs over a set of double coset representatives for $H\backslash G/B$; here we regard the $S$ on the left as an $H$-set.
\end{lemma}

This allows us to prove that, in the four cases, there is no point of
$\P^1(\F_l)$ fixed by all of $H$.  We apply the lemma with $G =
\PSL_2(\F_l)$, $S = \P^1(\F_l)$, and $B$ the stabiliser of $\infty$,
that is, the Borel subgroup. By the lemma, an orbit of size~$1$
corresponds to a double coset representative $g$ for which $H
\subseteq B^g$. But this inclusion is impossible, since each $H$
contains $D_4$ and $B$ does not. This finishes the proof.

\section{Proof of Proposition \ref{P4}}\label{sec:pf-prop}
Let $E/\Q(\sqrt{l})$ be a non-dihedral Hasse at $l$ curve. Then $l
\equiv 1\pmod{12}$, $\pmod{24}$ or~$\pmod{60}$, according as the
projective image of~$\bar{\rho}_{E,l}$ is $A_4$, $S_4$ or~$A_5$, by
Proposition~\ref{prop:ell-congruence}. However, there is the following
general result of David regarding the projective mod-$p$ image, for
which we are grateful to {Nicolas Billerey} for bringing to our
attention.  For $F$ a number field, and $p$ a prime, let
\[ e_p := \max_\mp\left\{e_\mp\right\},\]
where $e_\mp$ denotes the ramification index of the prime $\mp | p$.

\begin{fact}\cite[Lemme 2.4]{David}
For an elliptic curve defined over a number field $F$, the projective
mod-$p$ image contains an element of order at least~$\frac{p-1}{4e_p}$.
\end{fact}

Applying this with $F=\Q(\sqrt{l})$ and $p=l$, we see that
\begin{itemize}
\item
$A_4$ can occur only when $l \le 25$ and $l\equiv1\pmod{12}$, so only for
  $l=13$;
\item
$S_4$ can occur only when $l \le 33$ and $l\equiv1\pmod{24}$, so
  cannot occur;
\item
$A_5$ can occur only when $l \le 41$ and $l\equiv1\pmod{60}$, so
  cannot occur.
\end{itemize}
Thus only $A_4$ is possible, for the prime $l=13$.

\section{The Jacobian of $X_{S_4}(13)$}\label{sec:jacobians}
Over the complex numbers, there are precisely 3 modular curves of level 13 and genus~$3$; they are $\Xsplit(13)$, $\Xnonsplit(13)$, and $X_{S_4}(13)$; see for example the table of \cite{CumminsPauli}. Observe that all of these curves are defined over $\Q$ and are geometrically connected. 

In her recent work \cite{Baran1}, \cite{Baran2}, Burcu Baran proved, in two different ways, that the curves $\Xsplit(13)$ and $\Xnonsplit(13)$ are in fact $\Q$-isomorphic. Her first proof in \cite{Baran1} was computational; she computed models of both curves and showed that they give isomorphic curves. Her second proof was more conceptual, establishing that the Jacobians $\Jsplit(13)$ and $\Jnonsplit(13)$ are isomorphic, with an isomorphism preserving the canonical polarisation of both Jacobians; the Torelli theorem then gives the result. 

The $\Q$-points on $\Xsplit(13)$ have not yet been determined; in fact, as discussed in the final section of \cite{BPR}, $p=13$ is the \emph{only} prime $p$ for which the $\Q$-points on $\Xsplit(p)$ have \emph{not} yet been determined, and Baran's result, linking $\Xsplit(13)$ and $\Xnonsplit(13)$, may give some reason for why this $p=13$ case is so difficult: the determination of $\Q$-points on $\Xnonsplit(p)$ is known to be a difficult problem. 

Another reason for this difficulty is that $\Jsplit(13)(\Q)$ is likely to have Mordell-Weil rank 3, which equals the genus, so the method of Chabauty to determine the rational points does not apply. By likely, we mean that the analytic rank of this Jacobian is 3, so under the Birch-Swinnerton-Dyer conjecture, we would have that the Mordell-Weil rank is also 3.  

The curves $\Xsplit(13)$ and $X_{S_4}(13)$ are \emph{not} isomorphic, even over $\C$; this may be verified using the explicit models of both curves, by computing certain invariants of genus~$3$ curves and observing that they are different -- we are grateful to Jeroen Sijsling for carrying out this computation. 

Nevertheless, their Jacobians are isogenous. 

\begin{proposition}
The Jacobians \emph{$\Jsplit(13)$} and $J_{S_4}(13)$ of the modular curves \emph{$\Xsplit(13)$} and $X_{S_4}(13)$ are $\Q$-isogenous.
\end{proposition}

\def\PB{\pi^{-1}}
\begin{proof}
Let $G = \GL_2(\F_{13})$, $B$ the Borel subgroup of $G$, and for $K$
any subgroup of $\PGL_2(\F_{13})$, denote by $\pi^{-1}(K)$ the pullback of
$K$ to $G$. One first verifies (for example in \Magma) that there is a
$\Q[G]$-module isomorphism as follows:

\begin{align}
\label{thnxalx}
&2\Q[G/\Csplus] \oplus \Q[G/\PB(C_{13}\ltimes C_3)] \oplus \Q[G/\PB(C_{13} \ltimes C_4)] \cong \notag\\ &2\Q[G/\PB(S_4)] \oplus \Q[G/\PB(D_{26})] \oplus \Q[G/B].
\end{align}
For $R$ any $\Q$-algebra, apply the contravariant functor
$\Hom_{\Q[G]}(-,J(13)(R))$ to this formula; this yields, by a
well-known method of Kani and Rosen (\cite{KaniRosen}, but see also
\cite{EdixDeSmit}), the following $\Q$-isogeny between Jacobians of
modular curves of level~$13$: 

\begin{equation}
J_{S_4}^2 \oplus J_{\PB(D_{26})} \oplus J_{B} \quad\rightarrow\quad  \Jsplit^2 \oplus J_{\PB(C_{13}\ltimes C_3)} \oplus J_{\PB(C_{13} \ltimes C_4)};
\end{equation}
here we have, for simplicity, denoted the Jacobian of the modular curve $X_H(13)$ simply as $J_H$.

However, as may be checked by computing genera of these curves, most of these terms are zero, leaving us with a $\Q$-isogeny \[ J_{S_4}^2 \to \Jsplit^2.\]
Restricting this isogeny to the first component yields an isogeny between $J_{S_4}$ and its image in $\Jsplit^2$. This image must have dimension 3, and since $\Jsplit$ is simple over $\Q$ (as shown in Section 2 of \cite{Baran2}), the image is isogenous to $\Jsplit$. 
\end{proof}

\begin{remark}
One may still wonder whether $\Jsplit$ is isomorphic to $J_{S_4}$ or not. They are indeed not isomorphic; for if they were, then the arguments in Section 3 of \cite{Baran2} would apply, and we would conclude that the curves $\Xsplit$ and $X_{S_4}$ were isomorphic, which we know is not true. 
\end{remark}

\begin{remark}
With additional work one may show that there is a $\Q$-isogeny between $\Jsplit$ and $J_{S_4}$ of degree a power of 13, and furthermore that 13 must divide the degree of any isogeny. 
\end{remark}
\section{The evidence for Conjecture \ref{Co1}}\label{sec:evidence}
There can be no Hasse at $11$ curve over $\Q(\sqrt{-11})$, because $11$ is not congruent to~$1\pmod 4$ (see Proposition~\ref{P3}). Thus, we let $K$ be any other quadratic field. Sutherland's result (Proposition~\ref{Suther}) tells us that, if $E/K$ is a Hasse at $11$ curve over $K$, then it corresponds to a $K$-point on the modular curve $\Xsplit(11)$. A model for this curve, as well as an expression for the $j$-map $\Xsplit(11) \overset{j}\longrightarrow X(1)$, may be computed along the lines of that for $X_{S_4}(13)$; we obtain a singular projective model \[\Xsplit(11) : y^2 = 4X^6 - 4X^4 - 2X^3 + 2X^2 + \frac{3}{2}X + \frac{1}{4}.\] We used \Magma\ to search for $K$-points on this curve, for every quadratic field with absolute discriminant up to $10^7$, and evaluated the $j$-map at these points, giving many potential $j$-invariants of Hasse at $11$ curves over quadratic fields. 

Given such a $j$-invariant $j_0 \in K$, we considered the polynomial
$\Phi_{11}(X,j_0) \in K[X]$, that is, the classical modular polynomial
at~$11$, evaluated at $Y=j_0$.

\begin{proposition}
The pair $(11,j_0)$ is exceptional for $K$ if and only if
$\Phi_{11}(X,j_0) \in K[X]$ satisfies the following:
\begin{itemize}
\item
It has no linear factor over $K$;
\item
Modulo every prime $\mp$ in a density one set, it has a linear factor.
\end{itemize}
\end{proposition}

\begin{proof}
This is a direct consequence of the fact, proved by Igusa in \cite{Igusa}, that, for an elliptic curve $E$ over any field $F$, and an integer $N$ with char $F \nmid N$, the existence of a cyclic $F$-rational $N$-isogeny on $E$ is equivalent to $\Phi_N(X,j(E)) \in F[X]$ having a linear factor.
\end{proof}

We found that, for all of our potential $j$-invariants,
$\Phi_{11}(X,j_0)$ had many reductions with no linear factor -- too many to be of density zero. This suggested to us that Conjecture~\ref{Co1} should be true.

The results of Serre's paper \cite{SerreGal} imply the following.

\begin{proposition}
Let $K$ be a quadratic field. If $E/K$ is Hasse at 11, then either
\begin{itemize}
\item
$11$ ramifies in $K$; or
\item
$E$ has additive reduction at all  places $v$ of $K$ dividing~$11$.
\end{itemize}
\end{proposition}

\begin{proof}
If 11 is unramified in $K$ and $E$ has a place $v$ of good or multiplicative reduction above 11, then the results of \cite{SerreGal} (see in particular Section 4) gives the image of the inertia subgroup at $v$ of $G_K$ under $\bar{\rho}_{E,11}$, which in all cases is incompatible with the projective image being isomorphic with $D_{10}$. 
\end{proof}

We can also say, by part (a) of Proposition~\ref{Suther}, that $E/K$
is Hasse at 11 if and only if $H_{E,11} \cong D_{10}$, and so
corresponds to a $K$-point on the modular curve $X_{D_{10}}(11)$
parametrising such elliptic curves. This modular curve is the
$\Q(\sqrt{-11})$-twist of the more usual modular curve $X_0(121)$,
which, by Theorem 4.9 of \cite{Bars}, has only finitely many quadratic
points. Thus, we can say that there are only finitely many quadratic
fields over which a Hasse at 11 curve might exist. If we could
determine exactly which quadratic fields $K$ arise for the twist
$X_{D_{10}}(11)$ of $X_0(121)$, we could prove the conjecture by
determining the $K$-points on $\Xsplit(11)$, find the finite list of
potential $j$-invariants, and show that none of them yield $H_{E,5}
\cong D_{10}$ (this last step can be established by recent work of
Sutherland, who has implemented an algorithm to determine the mod-$p$
Galois image of any elliptic curve over any number field). The methods
of Freitas, Le~Hung and Siksek, who determine the quadratic points on
$X_0(N)$ for certain $N$ in \cite{FreitasSiksek}, may be of use to us
here.


\newcommand{\etalchar}[1]{$^{#1}$}

\end{document}